\documentclass[11pt,reqno]{amsart}

\usepackage{amsmath,amsthm,verbatim,amssymb,amsfonts,amscd, graphicx}
\usepackage{graphics}
\usepackage{geometry}
\theoremstyle{plain}

\newtheorem{defn}{Definition}[section]
\newtheorem{thm}[defn]{Theorem}
\newtheorem{cor}[defn]{Corollary}
\newtheorem{nota}[defn]{Notation}
\newtheorem{prop}[defn]{Proposition}

\newtheorem{remark}[defn]{Remark}
\newtheorem{lm}[defn]{Lemma}
\newtheorem{fact}[defn]{Fact}
\newtheorem{construction}[defn]{Construction}

\newtheorem*{1-mfd}{Proposition \ref{1-manifold}}
\newtheorem*{foliation}{Theorem \ref{foliation}}

\usepackage{extarrows}
\usepackage{mathrsfs}
\usepackage{galois}
\usepackage[colorlinks,citecolor = red, linkcolor=blue,hyperindex,CJKbookmarks]{hyperref}
\usepackage[all]{hypcap}
\usepackage[all]{xy}
\usepackage{tikz-cd}
\usepackage{latexsym}
\usepackage{subfigure}

\newcommand\blfootnote[1]{%
	\begingroup
	\renewcommand\thefootnote{}\footnote{#1}%
	\addtocounter{footnote}{-1}%
	\endgroup
}

\begin{document}
	
\title{Left orderability and taut foliations with one-sided branching}
\author{Bojun Zhao}
\address{Department of Mathematics, University at Buffalo, Buffalo, NY 14260, USA}
\email{bojunzha@buffalo.edu}
\maketitle

\begin{abstract}
	For a closed orientable irreducible $3$-manifold $M$ that admits 
	a co-orientable taut foliation with one-sided branching,
	we show that
	$\pi_1(M)$ is left orderable.
\end{abstract}


\section{Introduction}

The L-space conjecture is proposed by Boyer-Gordon-Watson (\hyperref[BGW]{[BGW]}) and 
by Juh\'asz (\hyperref[J]{[J]}),
which states that:
for every closed orientable irreducible
$3$-manifold $M$,
$M$ is a non-L-space if and only if
$\pi_1(M)$ is left orderable,
and if and only if
$M$ admits a co-orientable taut foliation.

In \hyperref[OS]{[OS]},
Ozsv\'ath and Szab\'o prove that
$M$ is a non-L-space if $M$ admits a co-orientable taut foliation
(see also \hyperref[Bo]{[Bo]}, \hyperref[KR]{[KR]}).
In \hyperref[G]{[G]},
Gabai proves that $M$ admits taut foliations if $M$ has positive first Betti number.
In \hyperref[BRW]{[BRW]},
Boyer, Rolfsen and Wiest prove that
$\pi_1(M)$ is left orderable if $b_1(M) > 0$.
It's known that the L-space conjecture holds for every graph manifold,
by the works of Boyer-Clay (\hyperref[BC]{[BC]}), 
Rasmussen (\hyperref[R]{[R]}), 
and 
Hanselman-Rasmussen-Rasmussen-Watson (\hyperref[HRRW]{[HRRW]}).

Taut foliations of $3$-manifolds can be distinguished to three types according to
the branching behavior for the leaf spaces
(of the pull-back foliations in their universal covers):
$\mathbb{R}$-covered, one-sided branching, and two-sided branching.
The generic case for taut foliations is
two-sided branching.
The $3$-manifolds that admit co-orientable $\mathbb{R}$-covered foliations are known to have
left orderable fundamental group.
In this paper,
we consider taut foliations with one-sided branching and
restrict the unknown case for the direction
``co-orientable taut foliation $\Longrightarrow$ left orderability of $\pi_1$'' to
two-sided branching.

\begin{thm}\label{foliation}
	Let $M$ be a connected, closed, orientable, irreducible $3$-manifold that 
	admits a co-orientable taut foliation with one-sided branching.
	Then $\pi_1(M)$ is left orderable.
\end{thm}

In the proof of Theorem \ref{foliation},
we provide an injective homomorphism of $\pi_1(M)$ to 
the group $\mathcal{G}_\infty$ (see Definition \ref{G-infty} for the definition),
which is shown to be
left orderable by Navas and
is shown to have no nontrivial homomorphism to $\text{Homeo}_+(\mathbb{R})$ by
Mann (\hyperref[Ma]{[Ma]}).

\subsection{Acknowledgements}

The author wishes to thank Professor Xingru Zhang for helpful discussions.
He also thanks Professor Cagatay Kutluhan for valuable comments.
And he would like to thank the anonymous referee for helpful comments that led to improvements in exposition.

\section{Preliminary}

Throughout this section,
$M$ will always be a connected, closed, orientable, irreducible $3$-manifold
and $p: \widetilde{M} \to M$ is the universal covering of $M$.

\subsection{The leaf space of a taut foliation}
Suppose that $\mathcal{F}$ is a taut foliation of $M$.
We will always adopt the following notations associated to $\mathcal{F}$:

\begin{nota}\rm
	Let $\widetilde{\mathcal{F}}$ be the pull-back foliation of $\mathcal{F}$ in $\widetilde{M}$.
	Let $L(\mathcal{F})$ be the leaf space of $\widetilde{\mathcal{F}}$.
	The action of $\pi_1(M)$ on $\widetilde{M}$ by deck transformations induces
	an action of $\pi_1(M)$ on $L(\mathcal{F})$,
	which we refer to as the \emph{$\pi_1$-action on $L(\mathcal{F})$}.
\end{nota}

Note that $L(\mathcal{F})$ is an orientable, connected, simply connected $1$-manifold,
which is second countable but possibly non-Hausdorff
(\hyperref[HR]{[HR]}, \hyperref[P]{[P]}).
In addition,
$L(\mathcal{F})$ is Hausdorff if and only if it is homeomorphic to $\mathbb{R}$.
Whenever $L(\mathcal{F})$ is not homeomorphic to $\mathbb{R}$,
$L(\mathcal{F})$ has some ``non-Hausdorff'' places described as follows.

Consider two distinct points $x,y \in L(\mathcal{F})$.
We say that $x, y$ are \emph{non-separated} if,
for any point $t \in L(\mathcal{F}) - \{x,y\}$,
$x, y$ are contained in the same component of $L(\mathcal{F}) - \{t\}$.
Equivalently,
$x, y$ are non-separated if
$x, y$ have no disjoint neighborhoods.
See Figure \ref{L} (a) for a collection of points in $L(\mathcal{F})$ which 
are pairwise non-separated.

We note that,
for three points $x, y, z \in L(\mathcal{F})$ such that
$x, y$ are non-separated and $y, z$ are non-separated,
it does not necessarily follow that $x, z$ are non-separated.
As discussed in \hyperref[RSS]{[RSS, Section 4]},
we can blow-up countably many points in $L(\mathcal{F})$ as in \hyperref[RSS]{[RSS, Appendix]}, 
to obtain a $1$-manifold $L^{'}$ so that
the relation ``$x \sim y$ if $x, y \in L^{'}$ are non-separated'' is transitive on $L^{'}$ and 
thus forms an equivalence relation on $L^{'}$.
The quotient space $L^{'} / \sim$ is Hausdorff and remains simply connected
(though $L^{'} / \sim$ is no longer a $1$-manifold), 
which is referred to as the \emph{Hausdorff tree associated to} $L(\mathcal{F})$
(\hyperref[RSS]{[RSS, Definition 4.6]}).

At last, 
we explain the behavior of the $\pi_1$-action on $L(\mathcal{F})$ in the case where
$\mathcal{F}$ is co-orientable.
We choose a co-orientation on $\mathcal{F}$,
then it induces a co-orientation on $\widetilde{\mathcal{F}}$
and therefore induces an orientation on $L(\mathcal{F})$.
In this case,
$\pi_1(M)$ acts on $L(\mathcal{F})$ via orientation-preserving homeomorphisms.

\begin{figure}\label{L}
	\centering
	\subfigure[]{
		\includegraphics[width=0.3\textwidth]{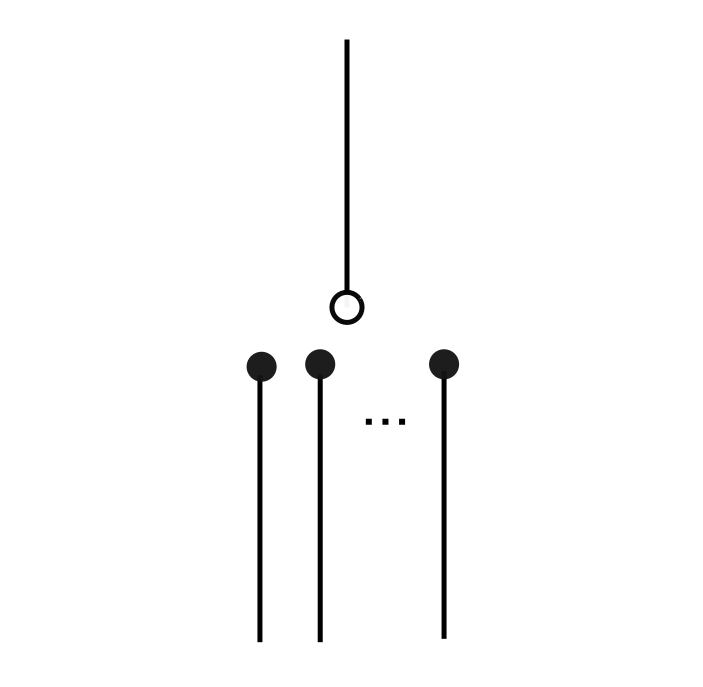}}
	\subfigure[]{
		\includegraphics[width=0.3\textwidth]{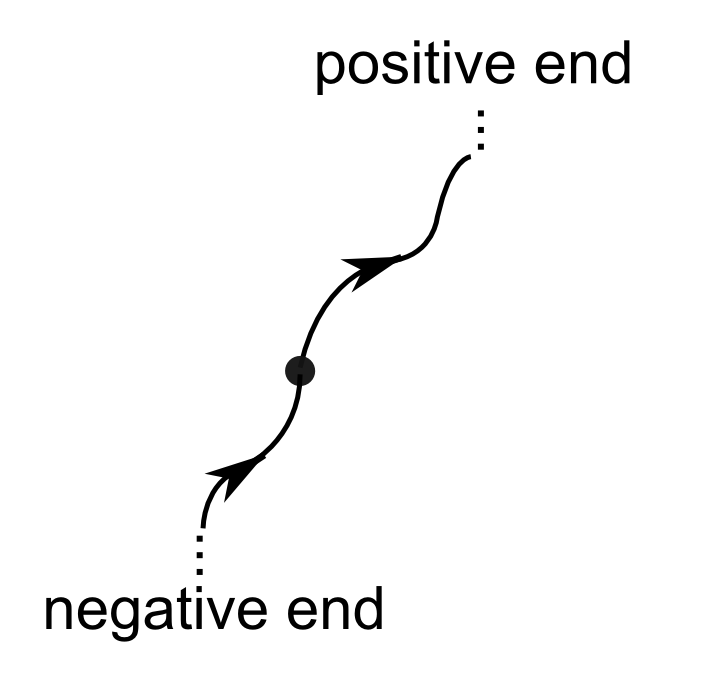}}
	\caption{Picture (a) describes a collection of points in $L(\mathcal{F})$ which are pairwise non-separated.
	Picture (b) describes the ends of $L(\mathcal{F})$:
every positive end of $L(\mathcal{F})$ can be represented by some positively oriented rays in $L(\mathcal{F})$,
and every negative end of $L(\mathcal{F})$ can be represented by some negatively oriented rays in $L(\mathcal{F})$.}
\end{figure}

\subsection{$1$-manifolds with one-sided branching}\label{subsection 2.1}
Let $L$ be an oriented, connected, simply connected $1$-manifold which 
is second countable but possibly non-Hausdorff.

By a \emph{ray} of $L$,
we mean an embedding $r: [0,+\infty) \to L$ such that 
there is no embedding $f: [0,+\infty) \to L$ with $f(0) = r(0)$ and
$r([0,+\infty)) \subsetneqq f([0,+\infty))$.

Let $\mathcal{R} = \{\text{rays of } L\}$.
And let $\sim$ be the relation on $\mathcal{R}$ such that,
for arbitrary two elements $r_1, r_2: [0,+\infty) \to L$ of $\mathcal{R}$,
$r_1 \sim r_2$ if there are 
$t_1,t_2 \in [0,+\infty)$ with $r_1([t_1,+\infty)) = r_2([t_2,+\infty))$.
It's not hard to see that ``$\sim$'' is transitive on $\mathcal{R}$,
and so ``$\sim$'' is an equivalence relation on $\mathcal{R}$.
We define
$$End(L) = \mathcal{R} / \sim,$$
and we call each element of $End(L)$ an \emph{end} of $L$.
For any $r \in \mathcal{R}$,
we will always denote by $[r]$ the end represented by $r$.

\begin{defn}\rm\label{proper embedding}
	Let $e: \mathbb{R} \to L$ be an embedding.
	We will always denote by 
	$e_+: [0, +\infty) \to L$ the restriction of $e$ to $[0, +\infty)$,
	and we will always denote by
	$e_-: [0, +\infty) \to L$ the map defined by
	$e_-(t) = e(-t)$ for any $t \in [0, +\infty)$.
	We call $e$ a \emph{proper embedding} if both of $e_+, e_-$ are rays of $L$.
\end{defn}

Given an embedded path $j: [0,1] \to L$,
a ray $r: [0, +\infty) \to L$,
or a proper embedding $e: \mathbb{R} \to L$,
each of $j, r, e$ is said to be \emph{positively oriented} (resp. \emph{negatively oriented}) if
the increasing orientation on its domain is consistent with (resp. \emph{opposite to}) 
the orientation on $L$.

\begin{defn}\rm
An end $t$ of $L$ is \emph{positive} (resp. \emph{negative}) if
there is a positively oriented ray (resp. negatively oriented ray)
$r: [0,+\infty) \to L$ with $[r] = t$.
See Figure \ref{L} (b).
\end{defn}

Let $G$ be a group acting on $L$ by homeomorphisms.
We note that this action induces an action of $G$ on $End(L)$.
In addition,
for any $g \in G$,
if $g: L \to L$ is an orientation-preserving homeomorphism,
then the endomorphism on $End(L)$ induced from $g$ takes positive ends to positive ends and
takes negative ends to negative ends.
Thus,
in the case that $G$ acts on $L$ via orientation-preserving homeomorphisms,
under the induced action of $G$ on $End(L)$,
the set of positive ends and the set of negative ends are both
$G$-invariant subsets of $End(L)$.

Below, 
we describe the branching behavior of $L$,
which basically follows from \hyperref[C2]{[C2, Chapter 4.7, page 169]}.
There is a strict partial order ``$\stackrel{_{L}}{>}$'' on $L$ defined as follows.
Let $u, v \in L(\mathcal{F})$.
Then
$u \stackrel{_{L}}{>} v$ if
$u \ne v$ and there is a positively oriented embedded path in $L$ from $v$ to $u$.
In addition,
$u, v$ are said to be \emph{comparable} if one of these three cases holds:
$u = v$, $u  \stackrel{_L}{>} v$, $u \stackrel{_L}{<} v$.
And $u, v$ are said to be \emph{incomparable} if $u, v$ are not comparable.

Compare the following definition with
\hyperref[C2]{[C2, Chapter 4.7, Definition 4.41 (2)]}.

\begin{defn}\rm\label{one-sided branching}
	$L$ is said to have \emph{one-sided branching} if either of the following statements holds:
	
	$\bullet$
	$L$ is not homeomorphic to $\mathbb{R}$,
	and moreover,
	for any $u,v \in L(\mathcal{F})$,
	there exists $s \in L(\mathcal{F})$ such that
	$s \stackrel{_{L}}{>} u, v$.
	In this case,
	$L$ is said to have \emph{(one-sided) branching in the negative direction}.
	
	$\bullet$
	$L$ is not homeomorphic to $\mathbb{R}$,
	and moreover,
	for any $u,v \in L(\mathcal{F})$,
	there exists $s \in L(\mathcal{F})$ such that
	$s \stackrel{_{L}}{<} u, v$.
	In this case,
	$L$ is said to have \emph{(one-sided) branching in the positive direction}.
\end{defn}

Now,
we provide a desciption for Definition \ref{one-sided branching} from the ends of $L$.
We consider $L$ as an order tree,
then by \hyperref[RS]{[RS, Definitions 3.2, 3.5, Theorem 3.6]},
if $u \stackrel{_{L}}{>} v$ for $u,v \in L$,
then there is a unique embedded path in $L$ from $v$ to $u$.
We denote this embedded path by $\gamma_L(v,u)$.

\begin{lm}\rm\label{one end}
	(a)
	$L$ has one-sided branching in the negative direction if and only if
	$L$ has exactly one positive end and more than one negative end.
	
	(b)
	$L$ has one-sided branching in the positive direction if and only if
	$L$ has exactly one negative end and more than one positive end.
\end{lm}
\begin{proof}
	To be convenient,
	we only prove (a).
	
	Suppose that $L$ has exactly one positive end and more than one negative end.
	Let $u, v \in L$.
	We choose two positively oriented rays $r_u, r_v: [0, +\infty) \to L$ with
	$r_u(0) = u$, $r_v(0) = v$.
	Because $L$ has exactly one positive end,
	$[r_u] = [r_v]$.
	So there exists $s \in r_u([0, +\infty)) \cap r_v([0, +\infty))$,
	and we have $s \stackrel{_{L}}{>} u, v$.
	Note that
	$L$ is not homeomorphic to $\mathbb{R}$ since it has more than one negative end.
	Hence
	$L$ has one-sided branching in the negative direction.
		
	Now suppose that $L$ has one-sided branching in the negative direction.
	We first claim that,
	for any two positively oriented rays $r_1, r_2: [0,+\infty) \to L$ with $r_1(0) = r_2(0)$,
	we have $[r_1] = [r_2]$.
	Let $a = r_1(0) = r_2(0)$.
	We assume otherwise $[r_1] \ne [r_2]$.
	Then there is
	$b \in r_1([0,+\infty)) - r_2([0,+\infty))$.
	If $b \stackrel{_L}{>} t$ for all $t \in r_2([0,+\infty))$,
	then $r_2([0,+\infty)) \subsetneqq \gamma_L(a, b)$,
	which is impossible.
	So there is $c \in r_2([0,+\infty))$ with
	$b \stackrel{_L}{\ngtr} c$ (here, $b \stackrel{_L}{\ngtr} c$ means that $b \stackrel{_L}{>} c$ doesn't hold).
	By Definition \ref{one-sided branching},
	there exists $d\in L$ such that
	$d \stackrel{_{L}}{>} b, c$.
	Note that $a \stackrel{_L}{<} b, c$ since $b \in r_1([0,+\infty))$, $c \in r_2([0,+\infty))$.
	Hence $a \stackrel{_L}{<} d$ and $b, c \in \gamma_L(a,d)$.
	It follows that
	either $b \in \gamma_L(a,c)$ or
	$b \in \gamma_L(c,d) - \{c\}$.
	As $b \stackrel{_L}{\ngtr} c$,
	we have $b \notin \gamma_L(c,d) - \{c\}$.
	So $b \in \gamma_L(a,c)$.
	Note that
	$c \in r_2([0,+\infty))$ implies
	$\gamma_L(a,c) \subseteq r_2([0,+\infty))$,
	and therefore
	$b \in r_2([0,+\infty))$.
	This contradicts the condition
	$b \in r_1([0,+\infty)) - r_2([0,+\infty))$ given above.
	Thus,
	$[r_1] = [r_2]$.
	
	Let $x, y \in L$ and 
	let $r_x, r_y: [0,+\infty) \to L$ be two positively oriented rays of $L$ with
	$x = r_x(0)$, $y = r_y(0)$.
	By Definition \ref{one-sided branching},
	there exists $t \in L$ such that
	$t \stackrel{_{L}}{>} x, y$.
	Let $r_t: [0,+\infty) \to L$ be a positively oriented ray of $L$ with $r_t(0) = t$.
	Then $\gamma_L(x,t) * r_t$ is a positively oriented ray of $L$ that starts at $x$,
	where $\gamma_L(x,t) * r_t$ denotes the concatenation of $\gamma_L(x,t)$ and $r_t$.
	Also,
	$\gamma_L(y,t) * r_t$ is a positively oriented ray of $L$ that starts at $y$.
	By the claim as above,
	we have
	$[r_x] = [\gamma_L(x,t) * r_t]$,
	$[r_y] = [\gamma_L(y,t) * r_t]$.
	Then
	$$[r_x] = [\gamma_L(x,t) * r_t] = [r_t] = [\gamma_L(y,t) * r_t] = [r_y].$$
	Thus,
	any two positively oriented rays in $L$ represent the same positive end of $L$,
	and therefore $L$ has a unique positive end.
	
	It remains to show that $L$ has more than one negative end.
	Suppose otherwise that $L$ has exactly one negative end.
	Since $L$ is not homeomorphic to $\mathbb{R}$,
	there exist $m, n \in L$ which are incomparable.
	Let $r^{+}_{m}, r^{+}_{n}: [0,+\infty) \to L$ be 
	two positively oriented rays of $L$ with $r^{+}_{m}(0) = m$, $r^{+}_{n}(0) = n$,
	and let $r^{-}_{m}, r^{-}_{n}: [0,+\infty) \to L$ be 
	two negatively oriented rays of $L$ with $r^{-}_{m}(0) = m$, $r^{-}_{n}(0) = n$.
	Then $[r^{+}_{m}] = [r^{+}_{n}]$,
	$[r^{-}_{m}] = [r^{-}_{n}]$.
	So there exist $q \in r^{+}_{m}([0,+\infty)) \cap r^{+}_{n}([0,+\infty))$, 
	$s \in r^{-}_{m}([0,+\infty)) \cap r^{-}_{n}([0,+\infty))$,
	and we have $q \stackrel{_L}{>} m, n$,
	$s \stackrel{_L}{<} m, n$.
	Thus, 
	$s \stackrel{_L}{<} q$ 
	and $m, n \in \gamma_L(s, q)$.
	It follows that $m, n$ are comparable,
	which is a contradiction.
	Therefore,
	$L$ has more than one negative end.
\end{proof}

\subsection{The three types of foliations}\label{subsection 2.2}

Suppose that $M$ admits a taut foliation $\mathcal{F}$.
We fix a co-orientation on $\widetilde{\mathcal{F}}$,
which induces an orientation on $L(\mathcal{F})$.
As described in \hyperref[C2]{[C2, Definition 4.41]},

\begin{defn}\rm\label{one-sided}
	$\mathcal{F}$ has exactly one of the following three types:
	
	(a)
	$\mathcal{F}$ is \emph{$\mathbb{R}$-covered} if
	$L(\mathcal{F})$ is homeomorphic to $\mathbb{R}$.
	
	(b)
	$\mathcal{F}$ has \emph{one-sided branching} if
	$L(\mathcal{F})$ has one-sided branching.
	
	(c)
	$\mathcal{F}$ has \emph{two-sided branching} if
	$\mathcal{F}$ is not $\mathbb{R}$-covered and does not have one-sided branching.
\end{defn}

For concreteness,
we describe Definition \ref{one-sided} (b), (c) by using the ends of $L(\mathcal{F})$.
We note that,
having more than one positive (or negative) end in $L(\mathcal{F})$ implies
the existence of infinitely many positive (or negative) ends in $L(\mathcal{F})$.
For the reader's convenience,
we include a proof of this property of $L(\mathcal{F})$.

\begin{lm}\rm
	(a)
	If $L(\mathcal{F})$ has more than one positive end,
	then $L(\mathcal{F})$ has infinitely many positive ends.
	
	(b)
	If $L(\mathcal{F})$ has more than one negative end,
	then $L(\mathcal{F})$ has infinitely many negative ends.
\end{lm}
\begin{proof}
	We only prove (a);
	the proof of (b) is entirely similar.
	Let $G = \pi_1(M)$,
	and we may consider $G$ as the group of deck transformations of $\widetilde{M}$.
	We first assume further that $\mathcal{F}$ is co-orientable and
	prove (a) under this assumption.
	
	For any $s \in L(\mathcal{F})$,
	we define 
	$$E_+(s) = \{[r] \mid r: [0,+\infty) \to L(\mathcal{F}) \text{ is a positively oriented ray with } r(0) = s\}.$$
	We claim that there exists $s \in L(\mathcal{F})$ with $|E_+(s)| > 1$,
	where $|E_+(s)|$ denotes the cardinality of $E_+(s)$.
	Now suppose otherwise that $|E_+(s)| = 1$ for all $s \in L(\mathcal{F})$.
	Note that for any $t_1, t_2 \in L(\mathcal{F})$ with $t_1 \stackrel{_{L(\mathcal{F})}}{>} t_2$,
	we always have $E_+(t_1) = E_+(t_2)$ since
	$E_+(t_1) \subseteq E_+(t_2)$ and $|E_+(t_1)| = |E_+(t_2)| = 1$.
	Because $L(\mathcal{F})$ has more than one positive end,
	there are $u, v \in L(\mathcal{F})$ with $E_+(u) \ne E_+(v)$.
	As described in \hyperref[GO]{[GO, Definition 6.9 (3), Proposition 6.10]},
	there is a finite sequence of points $u_0 = u, u_1, \ldots, u_{n-1}, u_n = v$ such that
	$u_i, u_{i+1}$ are comparable for every $0 \leqslant i \leqslant n-1$.
	It follows that $E_+(u_0), E_+(u_1), \ldots, E_+(u_n)$ are all equal,
	which is a contradiction.
	
	Thus,
	there exists $s \in L(\mathcal{F})$ with $|E_+(s)| > 1$.
	Let $r_1, r_2: [0, +\infty) \to L(\mathcal{F})$ be two positively oriented rays such that 
	$r_1(0) = r_2(0) = s$ but $[r_1] \ne [r_2]$.
	We choose $t \in r_1([0,+\infty)) - r_2([0,+\infty))$.
	Then $[r_2] \notin E_+(t)$.
	Because $\mathcal{F}$ is a taut foliation,
	there is a simple closed curve in $M$ transverse to $\mathcal{F}$ that
	intersects all leaves of $\mathcal{F}$.
	This condition implies that,
	there is a positively oriented path in $L(\mathcal{F})$ that starts at $t$,
	ends at certain point in the orbit of $t$ (under the $\pi_1$-action on $L(\mathcal{F})$),
	such that its interior contains certain point $s^{'}$ in the orbit of $s$ (under the $\pi_1$-action on $L(\mathcal{F})$).
	Let $g \in G$ with $s^{'} = g(s)$.
	We have $g(s) = s^{'} \stackrel{_{L(\mathcal{F})}}{>} t \stackrel{_{L(\mathcal{F})}}{>} s$.
	Since $\mathcal{F}$ is co-orientable,
	$g: L(\mathcal{F}) \to L(\mathcal{F})$ is an orientation-preserving homeomorphism,
	which implies that $g(a) \stackrel{_{L(\mathcal{F})}}{<} g(b)$ for any 
	$a, b \in L(\mathcal{F})$ with
	$a \stackrel{_{L(\mathcal{F})}}{<} b$.
	So
	$s \stackrel{_{L(\mathcal{F})}}{<} t \stackrel{_{L(\mathcal{F})}}{<} g(s) \stackrel{_{L(\mathcal{F})}}{<} 
	g^{2}(s) \stackrel{_{L(\mathcal{F})}}{<} 
	g^{3}(s) \stackrel{_{L(\mathcal{F})}}{<} \ldots$.
	Let $n \in \mathbb{N}$.
	Then
	
	(1)
	Since $[r_2] \notin E_+(t)$ and $g^{n}(s) \stackrel{_{L(\mathcal{F})}}{>} t$,
	$[r_2] \notin E_+(g^{n}(s))$.
	
	(2)
	For simplicity,
	we denote by $g^{n}([r_2])$ the image of $[r_2]$ under the endomorphism on $End(L(\mathcal{F}))$ induced from
	$g^{n}: L(\mathcal{F}) \to L(\mathcal{F})$.
	Since $g^{n}$ is orientation-preserving,
	$g^{n}([r_2])$ is a positive end of $L(\mathcal{F})$.
	Because $[r_2] \in E_+(s)$,
	we have 
	$g^{n}([r_2]) \in E_+(g^{n}(s))$.
	
	It follows that $[r_2], g^{n}([r_2])$ are distinct positive ends of $L(\mathcal{F})$.
	Consequently,
	$[r_2], g([r_2]), g^{2}([r_2]), \ldots$
	are positive ends of $L(\mathcal{F})$ that are distinct from each other.
	So $L(\mathcal{F})$ has infinitely many positive ends.
	
	Now assume that $\mathcal{F}$ is not co-orientable.
	Then $M$ has a double cover $M^{'}$ such that 
	$\mathcal{F}$ pulls-back to a co-orientable taut foliation of $M^{'}$ (denoted $\mathcal{F}^{'}$).
	Note that 
	$\widetilde{M}$ is also a universal cover of $M^{'}$ and
	$\widetilde{\mathcal{F}}$ is still the pull-back foliation of $\mathcal{F}^{'}$ in $\widetilde{M}$. 
	Thus $L(\mathcal{F})$ is homeomorphic to $L(\mathcal{F}^{'})$.
	We assign $L(\mathcal{F}^{'})$ an orientation induced from the orientation on $L(\mathcal{F})$.
	Similar to the above discussions,
	we can ensure that $L(\mathcal{F}^{'})$ has infinitely many positive ends.
	Thus $L(\mathcal{F})$ also has infinitely many positive ends.
\end{proof}

Combined with Lemma \ref{one end},

\begin{cor}
	(a)
	$\mathcal{F}$ has one-sided branching if and only if,
	$L(\mathcal{F})$ either has exactly one positive end and infinitely many negative ends
	or has exactly one negative end and infinitely many positive ends.
	
	(b)
	$\mathcal{F}$ has two-sided branching if and only if
	$L(\mathcal{F})$ has infinitely many positive ends and infinitely many negative ends.
\end{cor}

For taut foliations,
$\mathbb{R}$-covered is well studied,
two-sided branching is the generic case,
and one-sided branching has an intermediate role between them.

As noted in \hyperref[C1]{[C1, Subsection 2.1]},

\begin{fact}\rm
	If $\mathcal{F}$ has one-sided branching,
	then $\mathcal{F}$ is co-orientable.
\end{fact}

We refer the reader to \hyperref[C1]{[C1]} and \hyperref[F]{[F]} for 
more properties of taut foliations with one-sided branching.

There are many important examples of taut foliations with one-sided branching.
In \hyperref[Me]{[Me]},
Meigniez provides infinitely many taut foliations with one-sided branching in hyperbolic $3$-manifolds.
In \hyperref[C1]{[C1, Example 5.02]}, \hyperref[C2]{[C2, Example 4.43]},
Calegari describes the construction of 
taut foliations with one-sided branching in some closed $3$-manifolds obtained from 
taking finite branched covers over 
certain knots in some closed $3$-manifolds with $\mathbb{R}$-covered foliations.

\subsection{The blowing-up operations}\label{subsection 1.4}

For a set $C \subseteq \mathbb{R}$ consisting of countably many points,
Denjoy (\hyperref[D]{[D]}) introduces the 
\emph{blowing-up} operation for $C$ in $\mathbb{R}$.
As a result,
each point of $C$ is ``replaced'' by 
a new closed subinterval of $\mathbb{R}$ through the blowing-up operation. 
Similarly,
the blowing-up operation can be defined for a leaf of a taut foliation and
for a countable set of points in a possibly non-Hausdorff $1$-manifold.
Below,
we provide a brief description of these aspects along with
some references.

Suppose that $M$ admits a taut foliation $\mathcal{F}$,
and let $\lambda$ be a leaf of $\mathcal{F}$.
We refer the reader to \hyperref[C2]{[C2, Example 4.14]} for 
the blowing-up operation for $\lambda$ in $\mathcal{F}$.
In this operation,
the leaf $\lambda$ is replaced by 
a product subspace $\lambda \times I$ of $M$ foliated with leaves $\{\lambda \times \{t\} \mid t \in I\}$.
It's not hard to see that the resulting foliation is still taut.

Let $\mathcal{F}_0$ denote the taut foliation obtained from blowing-up $\lambda$ in $\mathcal{F}$.
Let $\widetilde{\lambda}$ be a component of $p^{-1}(\lambda)$,
which is also a leaf of $\widetilde{\mathcal{F}}$.
Let $G = \pi_1(M)$,
and we may consider $G$ as the group of deck transformations of $\widetilde{M}$.
Now consider the foliation $\widetilde{\mathcal{F}_0}$ in $\widetilde{M}$.
For each $g \in G$,
there is a product subspace $g(\widetilde{\lambda}) \times I$ of $\widetilde{M}$ foliated with leaves 
$\{g(\widetilde{\lambda}) \times \{t\} \mid t \in I\}$ of $\widetilde{\mathcal{F}_0}$,
which is also a component of $p^{-1}(\lambda \times I)$.

We refer the reader to \hyperref[RSS]{[RSS, 9. Appendix: Denjoy blow-ups]} for
the blowing-up operation for a countable set of points in a possibly non-Hausdorff $1$-manifold.
Now,
we consider every leaf of $\widetilde{\mathcal{F}}$ as a point of $L(\mathcal{F})$ and
consider every leaf of $\widetilde{\mathcal{F}_0}$ as a point of $L(\mathcal{F}_0)$.
Then $L(\mathcal{F}_0)$ is obtained from blowing-up the countable set 
$\{g(\widetilde{\lambda}) \mid g \in G\}$ of $L(\mathcal{F})$,
where every point $g(\widetilde{\lambda}) \in L(\mathcal{F})$ is replaced by the interval
$g(\widetilde{\lambda}) \times I \subseteq L(\mathcal{F}_0)$.

\subsection{The group $\mathcal{G}_\infty$}
We recall the definition of the group $\mathcal{G}_\infty$ below (see for example \hyperref[Ma]{[Ma, Definition 1.3]}):

\begin{defn}\rm\label{G-infty}
	Let $\mathcal{G}_{\infty} = \text{Homeo}_+(\mathbb{R}) / \sim$,
	where $\sim$ is the equivalence relation on $\text{Homeo}_+(\mathbb{R})$ defined by
	$g \sim f$ if there is $n \in \mathbb{R}$ such that
	the restrictions of $g,f$ to $[n,+\infty)$ are equal.
	Henceforth,
	for every $g \in \text{Homeo}_+(\mathbb{R})$,
	we will always denote by
	$[g]$ the image of $g$ under the quotient map
	$\text{Homeo}_+(\mathbb{R}) \to \mathcal{G}_\infty$.
	We assume that the multiplication on $\text{Homeo}_+(\mathbb{R})$ is given by
	the left group action,
	i.e. $fg = f \circ g$.
	Define $[f] \cdot [g] = [fg] = [f \circ g]$ for all
	$f,g \in \text{Homeo}_+(\mathbb{R})$.
	This multiplication is well-defined on $\mathcal{G}_\infty$ and
	makes $\mathcal{G}_{\infty}$ a group.
\end{defn}

Navas proves the following theorem,
by applying the criterion that,
a group $G$ is left orderable if and only if
for any finite subset $\{g_1, \ldots, g_n\}$ of $G$ that does not contain the identity,
there is $\epsilon_i \in \{-1,1\}$ ($1 \leqslant i \leqslant n$) such that
the semigroup generated by $g^{\epsilon_1}_1,\ldots,g^{\epsilon_n}_n$ does not contain the identity.
For a proof see \hyperref[Ma]{[Ma, Proposition 2.2]}.

\begin{thm}[Navas]\label{Navas}
	$\mathcal{G}_{\infty}$ is left orderable.
\end{thm}

\begin{remark}\rm
	Recall that a countable group is left orderable if and only if 
	it acts effectively on $\mathbb{R}$ via orientation-preserving homeomorphism,
	which implies that every countable subgroup of $\mathcal{G}_\infty$
	is a subgroup of $\text{Homeo}_+(\mathbb{R})$.
	However,
	Mann (\hyperref[Ma]{[Ma]}) proves that the cardinality of $\mathcal{G}_{\infty}$ is equal to the cardinality of 
	$\text{Homeo}_+(\mathbb{R})$,
	but there exists no nontrivial homomorphism $\mathcal{G}_{\infty} \to \text{Homeo}_+(\mathbb{R})$. 
	See \hyperref[Ma]{[Ma]} for more information about $\mathcal{G}_{\infty}$.
\end{remark}

\section{The proof of the main theorem}

\subsection{The group action on a $1$-manifold with one-sided branching}\label{subsection 3.1}

Let $L$ be an oriented, connected, simply connected $1$-manifold which 
is second countable but possibly non-Hausdorff,
and we assume that $L$ has one-sided branching.
Without loss of generality,
we may assume that $L$ has branching in the negative direction.
Let $G$ be a group acting on $L$ via orientation-preserving homeomorphisms.
Let $e: \mathbb{R} \to L$ be a positively oriented proper embedding. 
Then
$$[e_+] = \text{the positive end of } L,$$
$$[e_-] \in \{\text{negative ends of } L\}.$$
We fix this proper embedding $e$ throughout this subsection.

In this subsection,
we prove the following proposition:

\begin{prop}\label{1-manifold}
	There is a homomorphism $d: G \to \mathcal{G}_\infty$,
	satisfying that
	$d(h) \ne 1$ for all $h \in G$ with the following property:
	for arbitrary $n \in \mathbb{R}$,
	there is $m \in (n,+\infty)$ with $h(e(m)) \ne e(m)$.
\end{prop}

\begin{nota}\rm
	We assume that the action of $G$ on $L$ is a left group action,
	i.e. for arbitrary $f,g \in G$,
	we have $fg(x) = f(g(x))$ for every $x \in L$.
	Henceforth,
	for arbitrary functions $u: Y \to Z, v: X \to Y$,
	we will always denote by $uv$ the composition $u \circ v$.
	For example,
	given $f,g \in G$,
	$fge$ will always denote the function 
	$f \circ g \circ e: \mathbb{R} \stackrel{e}{\to} L \stackrel{g}{\to} L \stackrel{f}{\to} L$.
\end{nota}

\begin{figure}\label{e and ge}
	\includegraphics[width=0.6\textwidth]{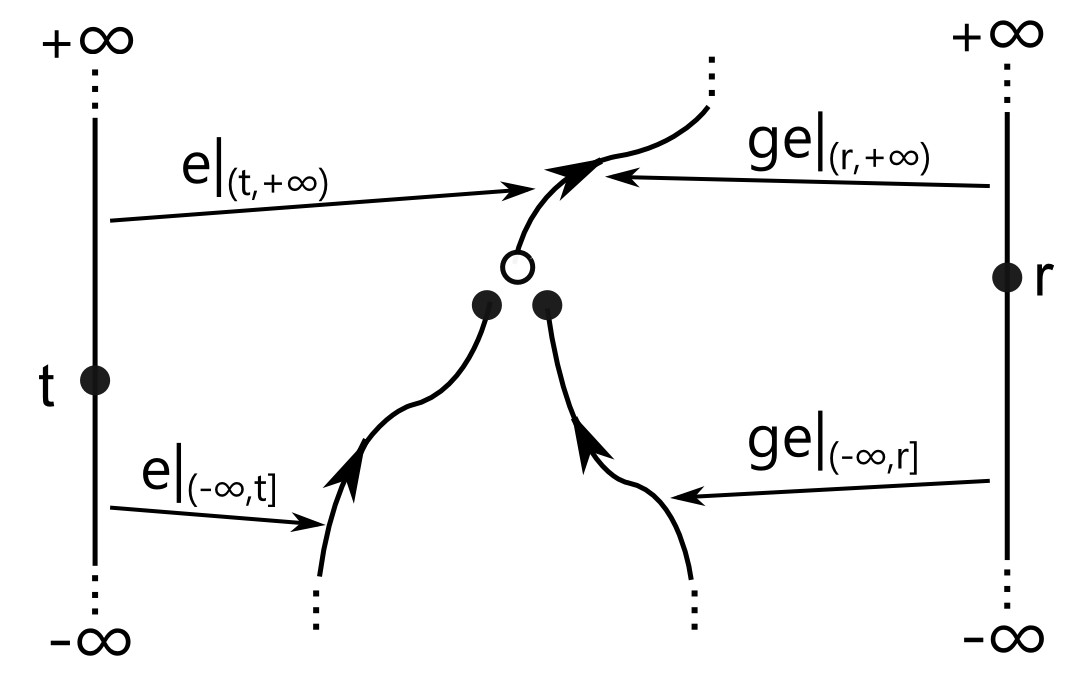}
	\caption{For $g \in G$ with $ge(\mathbb{R}) \ne e(\mathbb{R})$,
	there are $t,r \in \mathbb{R}$ such that 
	the image of $(t,+\infty)$ under $e$ is same as
    the image of $(r,+\infty)$ under $ge$,
    but $e(t) \ne ge(r)$}
\end{figure}

Let $g \in G$.
Since $g: L \to L$ is an orientation-preserving homeomorphism,
$ge: \mathbb{R} \to L$ is also a positively oriented proper embedding.
Thus $$[e_+] = [(ge)_+] = \text{the positive end of } L.$$
We can observe that one of the following two possibilities happens:

$\bullet$
$e(\mathbb{R}) = ge(\mathbb{R})$.

$\bullet$
$e(\mathbb{R}) \ne ge(\mathbb{R})$.
Then there are $J = (t,+\infty), K = (r,+\infty)$ for some $t,r \in \mathbb{R}$ such that
$$e(J) = ge(K) = e(\mathbb{R}) \cap ge(\mathbb{R}),$$
see Figure \ref{e and ge} for an example.

Thus,

\begin{fact}\rm\label{t exists}
Let $H = \{g_1,\ldots,g_n\} \subseteq G$ be a finite subset of $G$.
Let $$\mu = \bigcap_{1 \leqslant i \leqslant n}g_ie(\mathbb{R}).$$
Then $\mu \ne \emptyset$.
Moreover,
for each $1 \leqslant i \leqslant n$,
there is $J_i = [t_i,+\infty)$ for some $t_i \in \mathbb{R}$ such that
$g_ie(J_i) \subseteq \mu$.
\end{fact}

\begin{defn}\rm\label{d}
	Let $g \in G$.
	Let $\mu = e(\mathbb{R}) \cap ge(\mathbb{R})$.
	
	(a)
	We fix a homeomorphism $g_0: e(\mathbb{R}) \to ge(\mathbb{R})$ such that
	the restriction of $g_0$ to $\mu$ is identity.
	
	(b)
	Let $$g_* = e^{-1}g^{-1} _{0}ge: 
	\mathbb{R} \stackrel{e}{\to} e(\mathbb{R}) \stackrel{g}{\to} ge(\mathbb{R}) 
	\stackrel{g^{-1}_{0}}{\to} e(\mathbb{R}) \stackrel{e^{-1}}{\to} \mathbb{R}.$$
	We define $d(g) = [g_*] \in \mathcal{G}_\infty$.
\end{defn}

\begin{lm}
	For each $g \in G$, 
	$d(g)$ is independent of the choice of $g_0$.
\end{lm}
\begin{proof}
	Let $\mu = e(\mathbb{R}) \cap ge(\mathbb{R})$.
	By Fact \ref{t exists},
	there is $J = [n,+\infty)$ for some $n \in \mathbb{R}$ such that
	$e(J), ge(J) \subseteq \mu$.
	Then the restriction of $g^{-1}_{0}$ to $ge(J)$ is identity.
	Thus,
	$$g_* \mid_J = e^{-1}g^{-1} _{0}(ge\mid_J )= e^{-1} \cdot 1 \cdot (ge\mid_J) =
	(e^{-1}ge)\mid_J.$$
	So $g_* \mid_J$ is independent of the choice of $g_0$.
\end{proof}

In the following,
we prove that $d: G \to \mathcal{G}_\infty$ is a homomorphism.

\begin{lm}\label{homomorphism}
	Let $f,g \in G$.
	Then $d(fg) = d(f)d(g)$.
\end{lm}
\begin{proof}
	We have 
	\begin{align}
		f_*g_* &
		= (e^{-1}f^{-1}_{0} fe)(e^{-1}g^{-1}_{0}ge) \nonumber \\&
		= e^{-1}f^{-1}_{0}fg^{-1}_{0}ge, \nonumber
	\end{align}
	and
	$$(fg)_* = e^{-1}((fg)_0)^{-1}fge.$$
	Let $\mu = e(\mathbb{R}) \cap fe(\mathbb{R}) \cap ge(\mathbb{R}) \cap fge(\mathbb{R})$.
	By Fact \ref{t exists},
	there is $J = [t,+\infty)$ for some sufficiently large $t \in \mathbb{R}$ such that
	$$e(J), fe(J), ge(J),fge(J) \subseteq \mu.$$
	
	Notice that the restriction of $g^{-1}_{0}$ to $ge(J)$ is identity,
	since
	$$ge(J) \subseteq \mu \subseteq e(\mathbb{R}) \cap ge(\mathbb{R}).$$
	And the restriction of $f^{-1}_{0}$ to $fge(J)$ is identity,
	since
	$$fge(J) \subseteq \mu \subseteq e(\mathbb{R}) \cap fe(\mathbb{R}).$$
	So
	\begin{align}
		(f_*g_*) \mid_J &
		= e^{-1}f^{-1}_{0} fg^{-1}_{0}(ge\mid_J) \nonumber \\&
		= e^{-1}f^{-1}_{0}f \cdot 1 \cdot (ge\mid_J) \nonumber \\&
	    = e^{-1}f^{-1}_{0}(fge\mid_J) \nonumber \\&
		= e^{-1} \cdot 1 \cdot (fge\mid_J) \nonumber \\&
		= (e^{-1}fge) \mid_J. \nonumber
	\end{align}
	
	Also,
	the restriction of $((fg)_0)^{-1}$ to $fge(J) \subseteq \mu$ is identity,
	since 
	$$fge(J) \subseteq \mu \subseteq e(\mathbb{R}) \cap fge(\mathbb{R}).$$
	So
	\begin{align}
	(fg)_* \mid_J &
	= e^{-1}((fg)_0)^{-1}(fge\mid_J) \nonumber \\&
	= e^{-1} \cdot 1 \cdot (fge\mid_J) \nonumber \\&
	= (e^{-1}fge) \mid_J. \nonumber
   \end{align}
   It follows that
   $$(f_*g_*) \mid_J = (e^{-1}fge) \mid_J = (fg)_* \mid_J.$$
   By Definition \ref{G-infty},
   we have
   $$d(fg) = [(fg)_*] =[f_*][g_*] = d(f)d(g).$$
\end{proof}

To complete the proof of Proposition \ref{1-manifold},
it remains to show that

\begin{lm}\label{nontrivial}
	Let $h \in G$.
	Assume that for arbitrary $n \in \mathbb{R}$,
	there is $m \in (n,+\infty)$ with $he(m) \ne e(m)$.
	Then $d(h) \ne 1$.
\end{lm}
\begin{proof}
	Let $\mu = e(\mathbb{R}) \cap he(\mathbb{R})$.
	By Fact \ref{t exists},
	there is $J = [t,+\infty)$ for some $t \in \mathbb{R}$ such that
	$e(J),he(J) \subseteq \mu$.
	By Definition \ref{d},
	$d(h) = [e^{-1}h^{-1} _{0}he]$.
	Let $K = [r,+\infty)$ for some $r \in (t,+\infty)$.
	By our assumption,
	there is $m \in K$ such that
	$he(m) \ne e(m)$.
	Since $m \geqslant r > t$,
	$he(m)$ is contained in $\mu$,
	and thus $h^{-1} _{0}he(m) = he(m)$.
	Therefore,
	$$e^{-1}h^{-1} _{0}he(m) = e^{-1}he(m) \ne e^{-1}(e(m)) = m.$$
	So $e^{-1}h^{-1} _{0}he \nsim id$
	(where $\sim$ denotes the equivalence relation defined in Definition \ref{G-infty}).
	Therefore,
	$d(h) \ne 1$.
\end{proof}

Proposition \ref{1-manifold} now follows from 
Lemma \ref{homomorphism} and Lemma \ref{nontrivial}.

\subsection{The proof of Theorem \ref{foliation}}

Let $M$ be a connected, orientable, irreducible $3$-manifold.
Suppose that $M$ admits a co-oriented taut foliation $\mathcal{F}$ which has one-sided branching.
Let $G = \pi_1(M)$,
and let $p: \widetilde{M} \to M$ be the universal covering of $M$.

In this subsection,
we prove that

\begin{foliation}
	$G$ is left orderable.
\end{foliation}

Since $\mathcal{F}$ has one-sided branching,
$L(\mathcal{F})$ is a non-Hausdorff $1$-manifold with one-sided branching.
We may assume that $L(\mathcal{F})$ has an orientation induced from
the co-orientation on $\mathcal{F}$,
and that $L(\mathcal{F})$ has branching in the negative direction.
In the following,
for every $g \in G$ and $t \in L(\mathcal{F})$,
$g(t)$ will always denote the image of $t$ under 
the transformation $g: L(\mathcal{F}) \to L(\mathcal{F})$ given by
the $\pi_1$-action on $L(\mathcal{F})$.
And we will not distinguish the leaves of $\widetilde{\mathcal{F}}$ and
the points in $L(\mathcal{F})$.

    We first give a quick sketch of the proof of Theorem \ref{foliation} in this paragraph.
	We blow-up some leaf $\lambda$ of $\mathcal{F}$ to obtain 
	a new foliation $\mathcal{F}_0$ with one-sided branching.
	Then we construct an action
	$\{\alpha_g: L(\mathcal{F}_0) \to L(\mathcal{F}_0) \mid g \in G\}$ of $G$ on $L(\mathcal{F}_0)$
	such that
	some points in $L(\mathcal{F}_0)$ have trivial stabilizer.
	Considering the action $\{\alpha_g \mid g \in G\}$ on $L(\mathcal{F}_0)$ and
	choosing some positively oriented proper embedding 
	$e: \mathbb{R} \to L(\mathcal{F}_0)$ as in Subsection \ref{subsection 3.1},
	we can obtain an injective homomorphism $G \to \mathcal{G}_\infty$ by
	Proposition \ref{1-manifold}.

Now we give the details of the proof.
Let $\lambda$ be a leaf of $\mathcal{F}$, 
and
let $\widetilde{\lambda}$ be a component of $p^{-1}(\lambda)$.

\begin{fact}\rm\label{dense}
	Let $\rho: \mathbb{R} \to L$ be an arbitrary positively oriented proper embedding
	(then $[\rho_+]$ is the positive end of $L(\mathcal{F})$).
	Let $J = (n,+\infty)$ for some $n \in \mathbb{R}$.
	Then there is $g \in G$ such that
	$g(\widetilde{\lambda}) \in \rho(J)$.
\end{fact}
\begin{proof}
	Let $\widetilde{l} = \rho(m)$ for some $m > n$, 
	and let $l = p(\widetilde{l})$.
	Since $\mathcal{F}$ is taut,
	there is a path $\gamma$ in $M$ such that:
	$\gamma$ starts and ends at the same point in the leaf $l$,
	$\gamma$ intersects every leaf of $\mathcal{F}$,
	and $Int(\gamma)$ is a positively oriented transversal of $\mathcal{F}$.
	Let $\widetilde{\gamma}$ be a lift of $\gamma$ in $\widetilde{M}$ that
	starts at some point in the leaf $\widetilde{l}$,
	and we may consider $\widetilde{\gamma}$ as a path in $L$ that starts at the point $\widetilde{l}$.
	Notice that $\widetilde{\gamma}$ is positively oriented in $L$ and
	$L$ has branching in the negative direction,
	we have $\widetilde{\gamma} \subseteq \rho([m,+\infty)) \subseteq \rho(J)$.
	Since $\gamma$ has nonempty intersection with $\lambda$,
	there is $g \in G$ such that
	$g(\widetilde{\lambda}) \in \widetilde{\gamma}$.
	If follows that $g(\widetilde{\lambda}) \in \rho(J)$.
\end{proof}

We blow-up the leaf $\lambda$ of $\mathcal{F}$ to
obtain a new foliation $\mathcal{F}_0$ of $M$.
Then $L(\mathcal{F}_0)$ is obtained from
blowing-up $\{g(\widetilde{\lambda}) \mid g \in G\}$ in $L(\mathcal{F})$,
and thus $L(\mathcal{F}_0)$ is still a non-Hausdorff $1$-manifold with one-sided branching.
We assume that 
$L(\mathcal{F}_0)$ has an orientation induced from the orientation on
$L(\mathcal{F})$.
Now for every $g \in G$,
the point $g(\widetilde{\lambda}) \in L(\mathcal{F})$ is replaced by
an interval $g(\widetilde{\lambda}) \times I$.

Let $K = \{g \in G \mid g(\widetilde{\lambda}) = \widetilde{\lambda}\}$.
By Novikov's theorem (see for example \hyperref[C2]{[C2, Theorem 4.35]}), 
$\lambda$ is incompressible in $M$,
and thus the restriction of $p$ to $\widetilde{\lambda}$,
$p \mid_{\widetilde{\lambda}}: \widetilde{\lambda} \to \lambda$,
is a universal covering of $\lambda$.
As $K$ is the deck transformation group of the universal covering 
$p \mid_{\widetilde{\lambda}}: \widetilde{\lambda} \to \lambda$,
$K$ is isomorphic to $\pi_1(\lambda)$.
Since $\lambda$ is an orientable surface,
$K$ is a countable left orderable group or a trivial group (when $\lambda$ is a $2$-plane).
So there is an action $\phi: K \to \text{Homeo}_+(I)$ of $K$ on $I$ such that:
$\phi(g)(\frac{1}{2}) \ne \frac{1}{2}$ for every $g \in K - \{1\}$
(see \hyperref[C2]{[C2, Lemma 2.43, Remark]}).
Here,
we set $\phi$ to be the trivial homomorphism when $K$ is a trivial group.
For each left coset $gK$ ($g \in G$) of $K$ in $G$,
we fix an element $x_{gK} \in gK$.
And we set $x_K = 1 \in K$.

\begin{construction}\rm\label{alpha}
	For each $h \in G$,
	we define a map $\alpha_h: L(\mathcal{F}_0) \to L(\mathcal{F}_0)$ as follows:
	
	$\bullet$
	Suppose that $q \in L(\mathcal{F}_0) - \bigcup_{g \in G}(g(\widetilde{\lambda}) \times I)$.
	Then $q$ can be canonically identified with a point of $L(\mathcal{F})$
	(which we also denote by $q$).
	We define $\alpha_h(q) = h(q)$.
	
	$\bullet$
	Suppose that $q \in \bigcup_{g \in G}(g(\widetilde{\lambda}) \times I)$.
	Then there are $g \in G$, $t \in I$ such that
	$q = g(\widetilde{\lambda}) \times \{t\}$.
	We define
	$$\alpha_h(g(\widetilde{\lambda}) \times \{t\}) = 
	hg(\widetilde{\lambda}) \times \{\phi(x^{-1}_{hgK}hx_{gK})(t)\}.$$
\end{construction}

Since $x^{-1}_{hgK}hx_{gK} \in K$,
the map $\alpha_h$ is well-defined.
Notice that $\alpha_h$ takes $g(\widetilde{\lambda}) \times I$ to
$hg(\widetilde{\lambda}) \times I$ for every $g \in G$.
So $\alpha_h$ is an orientation-preserving homeomorphism.
Furthermore,

\begin{lm}
	$\{\alpha_g: L(\mathcal{F}_0) \to L(\mathcal{F}_0) \mid g \in G\}$ is
	an action of $G$ on $L(\mathcal{F}_0)$.
\end{lm}
\begin{proof}
	Let $h,r \in G$.
	It's clear that 
	$\alpha_{hr}(q) = \alpha_h\alpha_r(q)$ for every 
	$q \in L(\mathcal{F}_0) - \bigcup_{g \in G}(g(\widetilde{\lambda}) \times I)$.
	Now we choose
	$q \in \bigcup_{g \in G}(g(\widetilde{\lambda}) \times I)$.
	Let $g \in G$, $t \in I$ for which
	$q = g(\widetilde{\lambda}) \times \{t\}$.
	We have 
	\begin{align}
		\alpha_{hr}(q) &
		= \alpha_{hr}(g(\widetilde{\lambda}) \times \{t\}) \nonumber \\&
		= hrg(\widetilde{\lambda}) \times \{\phi(x^{-1}_{hrgK}hrx_{gK})(t)\} \nonumber
	\end{align}
and
\begin{align}
	\alpha_h\alpha_r(q) &
	= \alpha_h\alpha_r(g(\widetilde{\lambda}) \times \{t\}) \nonumber \\&
	= \alpha_h(rg(\widetilde{\lambda}) \times \{\phi(x^{-1}_{rgK}rx_{gK})(t)\}) \nonumber \\&
	= hrg(\widetilde{\lambda}) \times \{\phi(x^{-1}_{hrgK}hx_{rgK}x^{-1}_{rgK}rx_{gK})(t)\} \nonumber \\&
	= hrg(\widetilde{\lambda}) \times \{\phi(x^{-1}_{hrgK}hrx_{gK})(t)\}. \nonumber
\end{align}
Thus
$\alpha_{hr}(q) = \alpha_h\alpha_r(q)$.
Also,
we have
$\alpha_1(q) = q$ for every $q \in L(\mathcal{F}_0)$.
Therefore,
$\{\alpha_g: L(\mathcal{F}_0) \to L(\mathcal{F}_0) \mid g \in G\}$ is
an action of $G$ on $L(\mathcal{F}_0)$.
\end{proof}

\begin{lm}
	The point $\widetilde{\lambda} \times \{\frac{1}{2}\}$ has
	trivial stabilizer under $\{\alpha_g: L(\mathcal{F}_0) \to L(\mathcal{F}_0) \mid g \in G\}$.
\end{lm}
\begin{proof}
	Let $g \in G - \{1\}$.
	We have $$\alpha_g(\widetilde{\lambda} \times \{\frac{1}{2}\}) = 
	g(\widetilde{\lambda}) \times \{\phi(x^{-1}_{gK}gx_{K})(\frac{1}{2})\} =
	g(\widetilde{\lambda}) \times \{\phi(x^{-1}_{gK}g)(\frac{1}{2})\}.$$
	If $g \notin K$,
	then $g(\widetilde{\lambda}) \ne \widetilde{\lambda}$.
	If $g \in K - \{1\}$,
	then $\phi(x^{-1}_{gK}g)(\frac{1}{2}) = \phi(g)(\frac{1}{2}) \ne \frac{1}{2}$.
	So $\alpha_g(\widetilde{\lambda} \times \{\frac{1}{2}\}) \ne \widetilde{\lambda} \times \{\frac{1}{2}\}$
	for every $g \in G - \{1\}$.
\end{proof}

Let $\widetilde{\lambda_0} = \widetilde{\lambda} \times \{\frac{1}{2}\} \in L(\mathcal{F}_0)$.

\begin{defn}\rm\label{e}
Let $e: \mathbb{R} \to L(\mathcal{F}_0)$ be an arbitrary positively oriented proper embedding.
\end{defn}

\begin{lm}\label{limit}
	Let $J = (n,+\infty)$ for some $n \in \mathbb{R}$.
	Then there is $h \in G$ such that $\alpha_h(\widetilde{\lambda_0}) \in e(J)$.
\end{lm}
\begin{proof}
	This follows from Fact \ref{dense} and
	the fact that $L(\mathcal{F}_0)$ is obtained from
	blowing-up $\{g(\widetilde{\lambda}) \mid g \in G\}$ in $L(\mathcal{F})$
	(and every interval $g(\widetilde{\lambda}) \times I$ ($g \in G$)
	contains some images of $\widetilde{\lambda_0}$ under $\{\alpha_h \mid h \in G\}$).
\end{proof}

It follows that

\begin{cor}\label{trivial stabilizer}
	For every $n \in \mathbb{R}$,
	there is $m \in (n,+\infty)$ such that $\alpha_ge(m) \ne e(m)$ for all $g \in G - \{1\}$.
\end{cor}

Now
$(L(\mathcal{F}_0), \{\alpha_g \mid g \in G\}, e)$ can be considered as
the triple $(L,G,e)$ as given in Subsection \ref{subsection 3.1}.
Let $d: G \to \mathcal{G}_\infty$ be the homomorphism obtained by
Proposition \ref{1-manifold}.
Combined with Corollary \ref{trivial stabilizer},
we have

\begin{cor}
	For every $g \in G - \{1\}$,
	$d(g) \ne 1$.
\end{cor}

So $d$ is injective,
and thus $G$ is isomorphic to a subgroup of $\mathcal{G}_\infty$.
By Theorem \ref{Navas},
$\mathcal{G}_\infty$ is a left orderable group.
This completes the proof of Theorem \ref{foliation}.

\begin{remark}\rm
	Since $G$ is a countable left orderable group,
	$G$ acts nontrivially on $\mathbb{R}$.
	However,
	our approach does not give such an action.
	
	Theorem \ref{Navas} only
	ensures that a left-invariant order of $\mathcal{G}_\infty$ exists.
	There is no constructive proof for Theorem \ref{Navas} as far as we know
	(see some remarks in \hyperref[Ma]{[Ma, Section 1, Extension vs. realization]}).
	So we do not know a left-invariant order of $\mathcal{G}_\infty$ (or its dynamic realization),
	and 
	thus the homomorphism $d: G \to \mathcal{G}_\infty$ does not give
	a left-invariant order of $G$ or
	a nontrivial action of $G$ on $\mathbb{R}$.
\end{remark}

\end{document}